\documentclass[10pt,a4paper]{article}
\usepackage{amsmath}
\usepackage{amsthm}
\usepackage{amsfonts}
\usepackage{latexsym}

\newcounter{alphthm}
\newtheorem{thm}{Theorem}[section]
\newtheorem{lem}[thm]{Lemma}

\theoremstyle{definition}

\newcommand{\be}{\begin{equation}}
\newcommand{\ee}{\end{equation}}
\newcommand{\ben}{\begin{enumerate}}
\newcommand{\een}{\end{enumerate}}
\newcommand{\pa}{{\partial}}

\title{Ricci Flow Equation on $(\alpha, \beta)$-Metrics}
\author{A. Tayebi, E. Peyghan and B. Najafi}
\begin{document}

\maketitle
\begin{abstract}
In this paper, we study the class of Finsler metrics, namely $(\alpha, \beta)$-metrics, which satisfies the un-normal or normal Ricci flow equation.
\end{abstract}

\section{Introduction}
In 1982, R. S. Hamilton for a
Riemannian metric $g_{ij}$ introduce the  following geometric evolution equation
\[
\frac{d}{dt}(g_{ij}) = -2 Ric_{ij}, \ \  g(t = 0)= g_0,
\]
where $Ric_{ij}$ is the Ricci curvature tensor and is known as the un-normalised Ricci flow in Riemannian geometry \cite{ham1}. Hamilton  showed that there is a unique solution to this equation for an
arbitrary smooth metric on a closed manifold over a sufficiently short time. He also showed that Ricci flow preserves positivity of the
Ricci curvature tensor in three dimensions and the curvature operator in all
dimensions \cite{H}. The Ricci flow theory related geometric analysis and various applications  became one of the most intensively
developing branch of modern mathematics \cite{caozhu,ham1,kleiner,rbook,SR}. The most important achievement of
this theory was the proof of W. Thurston's geometrization conjecture by
G. Perelman \cite{gper1,gper2,gper3}. The main results on Ricci flow
evolution were proved originally for (pseudo) Riemannian and K\"{a}hler
geometries. Thus the Ricci flow theory became a very powerful method in understanding
the geometry and topology of Riemannian and K\"{a}hlerian manifolds.

On the other hand, Finsler geometry is a natural extension of Riemannian geometry without  quadratic restriction \cite{TP2}. But it is not simple to define Ricci flows of mutually compatible fundamental geometric structures on Finsler manifolds. The problem
of constructing the Finsler-Ricci flow theory contains a number of new conceptual
and fundamental issues on compatibility of geometrical and physical
objects and their optimal configurations. The same equation can be used in the Finsler setting, since both the fundamental tensor  $g_{ij}$ and Ricci tensor $Ric_{ij}$ have been generalized to that broader framework,
albeit gaining a $y$ dependence in the process \cite{Ba}\cite{TP}. However, there are some reasons why we shall refrain from doing so: (i) not every symmetric covariant 2-tensor $g_{ij}(x, y)$ arises from a Finsler metric $F(x, y)$; (ii) there is more than one geometrical context in which $g_{ij}$ makes sense. Thus, Bao  called this equation as an un-normalised Ricci flow for Finsler geometry. Using the elegance work of Akbar-Zadeh in \cite{A}, Bao  proposed  the following normalised  Ricci flow equation for Finsler metrics
\begin{equation}
\frac{d}{dt}\log F= -R +\frac{1}{Vol(SM)}\int_{SM}R\ dV,\ \ \  F(t=0)=F_0,
\end{equation}
where the underlying manifold $M$ is compact \cite{Ba}.

In a series of papers,  Vacaru studied Ricci
flow evolutions of geometries and physical models of gravity with symmetric
and nonsymmetric metrics and geometric mechanics, when the field equations are subjected to nonholonomic
constraints and the evolution solutions, mutually transform as Riemann and Finsler geometries \cite{vricci1,vricci2,vricci3,vricci4,vfract1,vrfsol2,vrfsol3}.

It is remarkable that, Chern had asked  whether every smooth manifold admits a Ricci-constant Finsler metric? The weaker case of this question is that whether every smooth manifold admits a Einstein Finsler metric? His question has already been settled in the affirmative for dimension 2 because, by a construction of Thurston's, every Riemannian metric on a two-dimensional manifold admits a complete Riemannian metric of constant Gaussian curvature.

\section{Preliminaries}
Let $M$ be an $n$-dimensional $ C^\infty$ manifold. Denote by $T_x M $
the tangent space at $x \in M$, and by $TM=\cup _{x \in M} T_x M $
the tangent bundle of $M$.

A Finsler metric on $M$ is a function $ F:TM
\rightarrow [0,\infty)$ which has the following properties:\\
(i) $F$ is $C^\infty$ on $TM_{0}:= TM \setminus \{ 0 \}$;\\
(ii) $F$ is positively 1-homogeneous on the fibers of tangent bundle $TM$,\\
(iii) for each $y\in T_xM_0$, the following  form $g_y$ on $T_xM$  is positive definite,
\[
g_{y}(u,v):={1 \over 2} \left[  F^2 (y+su+tv)\right]|_{s,t=0}, \ \
u,v\in T_xM.
\]
For a Finsler metric $F=F(x,y)$ on a manifold $M$, the spray ${\bf G}=y^i\frac{\partial}{\partial x^i}-2 G^i \frac{\partial}{\partial y^i}$ is a vector field on $TM$, where $G^i=G^i(x,y)$ are defined by
\[
G^i = {g^{il}\over 4} \Big\{ [F^2]_{x^k y^l}y^k - [F^2]_{x^l} \Big\},\label{Gi}
\]

\bigskip

Let  $x\in M$ and $F_x:=F|_{T_xM}$.  To measure the non-Euclidean feature of $F_x$, define ${\bf C}_y:T_xM\otimes T_xM\otimes T_xM\rightarrow \mathbb{R}$ by
\[
{\bf C}_{y}(u,v,w):={1 \over 2} \frac{d}{dt}\left[{\bf g}_{y+tw}(u,v)
\right]|_{t=0}, \ \ u,v,w\in T_xM,
\]
The family ${\bf C}:=\{{\bf C}_y\}_{y\in TM_0}$  is called the Cartan torsion. It is well known that ${\bf{C}}=0$ if and only if $F$ is Riemannian. For $y\in T_x M_0$, define  mean Cartan torsion ${\bf I}_y$ by ${\bf I}_y(u):=I_i(y)u^i$, where $I_i:=g^{jk}C_{ijk}$, $C_{ijk}=\frac{1}{2}\frac{\pa g_{ij}}{\pa y^k}$ and $u=u^i\frac{\partial}{\partial x^i}|_x$. By Deicke's  Theorem, $F$ is Riemannian  if and only if ${\bf I}_y=0$.

Regarding the Cartan tensors of these metrics,  M. Matsumoto  introduced the notion of C-reducibility and proved that any Randers metric  $F=\alpha+\beta$ and Kropina metric $F=\alpha^2/\beta$ are C-reducible, where $\alpha=\sqrt{a_{ij}y^iy^j}$  is  a Riemannian metric and  $\beta=b_i(x)y^i$ is a 1-form on $M$. Matsumoto-H\={o}j\={o} proved that the converse is true \cite{MH}. Furthermore, by considering Kropina and Randers metrics, Matsumoto introduced the notion of $(\alpha, \beta)$-metrics \cite{Mat5}. An $(\alpha, \beta)$-metric is a Finsler metric on $M$ defined by $F:=\alpha\phi(s)$, where $s=\beta/\alpha$,  $\phi=\phi(s)$ is a $C^\infty$ function on the $(-b_0, b_0)$ with certain regularity, $\alpha$ is a Riemannian metric and $\beta$ is a 1-form on $M$.

In \cite{Mat4},  Matsumoto-Shibata introduced the notion of semi-C-reducibility by considering the form of Cartan torsion of a non-Riemannian $(\alpha,\beta)$-metric on a manifold $M$ with dimension $n\geq 3$. A Finsler metric is called semi-C-reducible if its Cartan tensor is given by
\[
C_{ijk}={\frac{p}{1+n}}\{h_{ij}I_k+h_{jk}I_i+h_{ki}J_j\}+\frac{q}{C^2}I_iI_jI_k,
\]
where $p=p(x,y)$ and $q=q(x,y)$ are scalar function on $TM$, $h_{ij}:=g_{ij}-F^{-2}y_iy_j$ is the angular metric and $C^2=I^iI_i$.  If $q=0$, then $F$ is just C-reducible  metric.

\section{Ricci Flow Equation}
In 1982, R. S. Hamilton introduce the  following geometric evolution equation
\[
\frac{d}{dt}(g_{ij}) = -2 Ric_{ij}, \ \  g(t = 0)= g_0
\]
which is known as the un-normalised Ricci flow in Riemannian geometry \cite{ham1}. The same equation can be used in the Finsler setting, since both the fundamental tensor  $g_{ij}$ and Ricci tensor $Ric_{ij}$ have been generalized to that broader framework,
albeit gaining a $y$ dependence in the process. However, there are some reasons why we shall refrain from doing so: (i) Not every symmetric covariant 2-tensor $g_{ij}(x, y)$ arises from a Finsler metric $F(x, y)$; (ii) There is more than one geometrical context in which $g_{ij}$ makes sense. Thus, Bao  called this equation as an un-normalised Ricci flow for Finsler geometry.

Professor Chern had asked, on several occasions, whether every smooth manifold
admits a Ricci-constant Finsler metric. It is hoped that the Ricci flow in
Finsler geometry eventually proves to be viable for addressing Chern's question.
How to formulate and generalize these constructions for non-Riemannian
manifolds and physical theories is a challenging topic in mathematics and
physics. Bao studied Ricci flow equation in Finsler spaces \cite{Ba}. In the following
a scalar Ricci flow equation is introduced according to the Bao's paper.

A deformation of Finsler metrics means a 1-parameter family of metrics
$g_{ij}(x, y, t)$, such that $t\in [-\epsilon, \epsilon]$ and $\epsilon > 0$ is sufficiently small. For such a
metric $\omega= u_idx^i$, the volume element as well as the connections attached to
it depend on $t$. The same equation can be used in the Finsler setting. We can
also use another Ricci flow equation instead of this tensor evolution equation \cite{Ba}. By contracting $\frac{d}{dt}g_{ij}=-2Ric_{ij}$ with $y^i$ and $y^j$ gives, via Euler's theorem,
we get
\[
\frac{\partial F^2}{\partial t}=-2F^2 R,
\]
where $R=\frac{1}{F^2} Ric$. That is,
\[
d \log F=-R, \ \ F(t=0)= F_0.
\]
This scalar equation directly addresses the evolution of the Finsler metric $F$,
and makes geometrical sense on both the manifold of nonzero tangent vectors
$TM_0$ and the manifold of rays. It is therefore suitable as an un-normalized
Ricci flow for Finsler geometry.

\section{Un-Normal Ricci Flow Equation on $(\alpha, \beta)$-Metrics }
Here,  we study $(\alpha, \beta)$-metrics satisfying  un-normal Ricci flow equation and prove the following.

\begin{thm}\label{thm1}
Let $(M, F)$ be a Finsler manifold of dimension $n\geq 3$. Suppose that $F=\Phi(\frac{\beta}{\alpha})\alpha$ be an $(\alpha, \beta)$-metric on $M$.  Then every deformation $F_t$ of the metric $F$  satisfying un-normal Ricci flow equation is an  Einstein metric.
\end{thm}

To prove the Theorem \ref{thm1},  we need the following.

\bigskip

\begin{lem}\label{Lem1}
Let $F_t$ be a deformation of an $(\alpha, \beta)$-metric $F$ on manifold $M$ of dimension $n\geq 3$. Then the variation of Cartan tensor is given by following
\begin{eqnarray}
C'_{ijk}I^i I^j I^k=- \frac{2R(1+nq)}{1+n}||I||^4-\frac{1}{2}F^2R_{,i,j,k}I^i I^j I^k- 3||I||^2 I^mR_{,m}\label{Ric3}
\end{eqnarray}
where $||I||^2=I_mI^m$.
\end{lem}
\begin{proof}
First assume that $F_t$ be a deformation of a Finsler metric on a two-dimensional manifold $M$ satisfies
Ricci flow equation, i.e.
\begin{equation}
\frac{d}{dt}g_{ij}:= g'_{ij}=-2Ric_{ij}, \ \ \ d \log F:=\frac{F'}{F}=-R,
\end{equation}
where $R=\frac{1}{F^2} Ric$. By definition of Ricci tensor, we have
\begin{eqnarray}
\nonumber Ric_{ij}\!\!\!\!&=&\!\!\!\!\ \frac{1}{2}[RF^2]_{y^iy^j}\\
\!\!\!\!&=&\!\!\!\!\ R g_{ij}+\frac{1}{2}F^2R_{,i,j}+R_{,i}y_j+R_{,i}y_i \label{R}
\end{eqnarray}
where $R_{,i}=\frac{\partial R}{\partial y^i}$ and $R_{,i,j}=\frac{\partial^2 R}{\partial y^i\partial y^j}$. Taking a vertical derivative of (\ref{R}) and using $y_{i,j}=g_{ij}$ and $FF_k=y_k$ yields
\begin{eqnarray}
\nonumber Ric_{ij,k}= 2RC_{ijk}+ \frac{1}{2}F^2R_{,i,j,k}\!\!\!\!&+&\!\!\!\!\ \{g_{jk}R_{,i}+g_{ij}R_{,k}+g_{ki}R_{,j}\}\\
\!\!\!\!&+&\!\!\!\!\ \{R_{,j,k}y_i+R_{,i,j}y_k+R_{,k,i}y_j\}.\label{Ric1}
\end{eqnarray}
Contracting (\ref{Ric1}) with $I^i I^j I^k$ and using $y_iI^i=y^iI_i=0$ implies that
\begin{eqnarray}
Ric_{ij,k}I^i I^j I^k= 2RC_{ijk}I^i I^j I^k+\frac{1}{2}F^2R_{,i,j,k}I^i I^j I^k+ 3||I||^2 I^mR_{,m}.\label{Ric}
\end{eqnarray}
The Cartan tensor of an $(\alpha, \beta)$-metric on a $n$-dimensional manifold $M$ is given by
\begin{equation}\label{Cartan}
C_{ijk}={\frac{p}{1+n}}\{h_{ij}I_k+h_{jk}I_i+h_{ki}J_j\}+\frac{q}{||I||^2}I_iI_jI_k,
\end{equation}
where $p=p(x,y)$ and $q=q(x,y)$ are scalar function on $TM$ with $p+q=1$. Multiplying  (\ref{Cartan}) with $I^i I^j I^k$ yields
\begin{equation}\label{C}
C_{ijk}I^i I^j I^k=({\frac{p}{1+n}}+q )||I||^4=\frac{1+nq}{1+n}||I||^4.
\end{equation}
Then by (\ref{Ric}) and (\ref{C}), we get
\begin{eqnarray}
Ric_{ij,k}I^i I^j I^k= \frac{2R(1+nq)}{1+n}||I||^4+\frac{1}{2}F^2R_{,i,j,k}I^i I^j I^k+ 3||I||^2 I^mR_{,m}.\label{Ric2}
\end{eqnarray}
On the other hand, since $F_t$ satisfies Ricci flow equation then
\begin{eqnarray}
C'_{ijk}=\frac{1}{2}\frac{\partial g'_{ij}}{\partial y^k}=\frac{1}{2}\frac{\partial(-2Ric_{ij})}{\partial y^k}=-Ric_{ij,k}.\label{Car}
\end{eqnarray}
By (\ref{Ric2}) and (\ref{Car}) we get (\ref{Ric3}).
\end{proof}

\bigskip

\begin{lem}\label{Lem2}
Let $F_t$ be a deformation of an $(\alpha, \beta)$-metric $F$ on a $n$-dimensional manifold $M$. Then $C'_{ijk}I^iI^jI^k$ is a factor of $||I||^2$.
\end{lem}
\begin{proof}
Since $g^{ij}g_{jk}=\delta^i_k$,  then we have
\[
0=(g^{ij}g_{jk})'=g'^{ij}g_{jk}+g^{ij}g'_{jk}=g'^{ij}g_{jk}-2g^{ij}Ric_{jk},
\]
or equivalently $g'^{ij}g_{jk}=2g^{ij}Ric_{jk}$ which contracting it  with  $g^{lk}$ implies that
\begin{equation}
g'^{il}=2Ric^{il}.
\end{equation}
Then we have
\begin{eqnarray}
\nonumber I'_i=(g^{jk}C_{ijk})' \!\!\!\!&=&\!\!\!\!\ (g^{jk})'C_{ijk}+g^{jk}C'_{ijk}\\
\nonumber \!\!\!\!&=&\!\!\!\!\ 2Ric^{jk}C_{ijk}-g^{jk}Ric_{jk,i}\\
\nonumber \!\!\!\!&=&\!\!\!\!\ 2Ric^{jk}g_{jk,i}-(g^{jk}Ric_{jk})_{,i}+g^{jk}_{\ \ ,i}Ric_{jk}\\
\!\!\!\!&=&\!\!\!\!\ -(g^{jk}Ric_{jk})_{,i}=-\rho_{i}
\end{eqnarray}
where $\rho:=g^{jk}Ric_{jk}$ and $\rho_i=\frac{\partial \rho}{\partial y^i}$. Thus
\begin{eqnarray}
\nonumber I'^i=(g^{ij}I_j)' \!\!\!\!&=&\!\!\!\!\ (g^{ij})'I_j+g^{ij}I'_j\\
\nonumber \!\!\!\!&=&\!\!\!\!\ 2Ric^{ij}I_j-g^{ij}\rho_{j}\\
\!\!\!\!&=&\!\!\!\!\ 2Ric^{ij}I_j-\rho^{i}.
\end{eqnarray}
The variation of $y_i:=FF_{y^i}$ with respect to $t$ is given by
\[
y'_i=-2Ric_{im}y^m.
\]
Therefore, we can compute the variation of angular metric $h_{ij}$ as follows
\begin{eqnarray}
\nonumber h'_{ij}=(g_{ij}-F^{-2}y_iy_j)' \!\!\!\!&=&\!\!\!\!\ -2Ric_{ij}-2F^{-2}Ry_iy_j+2F^{-2}(Ric_{im}+Ric_{jm})y^m\\
\nonumber\!\!\!\!&=&\!\!\!\!\ -2Ric_{ij}+2R(h_{ij}-g_{ij})+2(Ric_{im}\ell_j+Ric_{jm}\ell_i)\ell^m,
\end{eqnarray}
where $\ell_i:=F^{-1}y_i$. Thus
\begin{equation}
h'_{ij}=2Rh_{ij}-2Rg_{ij}-2Ric_{ij}+2(Ric_{im}\ell_j+Ric_{jm}\ell_i)\ell^m.
\end{equation}
Now, we consider the variation of Cartan tensor
\begin{eqnarray}\label{mainCar}
\nonumber C'_{ijk} \!\!\!\!&=&\!\!\!\!\ \big\{\frac{p}{1+n}[h_{ij}I_k+h_{jk}I_i+h_{ki}J_j]+\frac{q}{||I||^2}I_iI_jI_k\big\}'\\
\nonumber \!\!\!\!&=&\!\!\!\!\ \frac{p'}{1+n}[h_{ij}I_k+h_{jk}I_i+h_{ki}J_j]+\frac{q'}{||I||^2}I_iI_jI_k\\
\!\!\!\!&+&\!\!\!\!\ \frac{p}{1+n}\Big[h_{ij}I_k+h_{jk}I_i+h_{ki}J_j\Big]'+q\Big[\frac{1}{||I||^2}I_iI_jI_k\Big]'
\end{eqnarray}
We have
\begin{equation}\label{-1}
[\frac{1}{||I||^2}I_iI_jI_k]'=\frac{-(I'^mI_m+I^mI'_m)}{||I||^2}C_{ijk} -\frac{1}{||I||^2}(\rho_{i}I_jI_k+\rho_{j}I_iI_k+\rho_{k}I_iI_j)
\end{equation}
Multiplying (\ref{-1}) with $I^i I^j I^k$ implies that
\begin{eqnarray}\label{1}
\nonumber[\frac{1}{||I||^2}I_iI_jI_k]'I^iI^jI^k\!\!\!\!&=&\!\!\!\!\ - \big[\frac{(nq+1)(I'^mI_m+I^mI'_m)}{(n+1)||I||^2}+3\rho_mI^m\big]||I||^2\\
\!\!\!\!&=&\!\!\!\!\ \big[\frac{2(nq+1)(\rho^mI_m-Ric^{pq}I_pI_q)}{(n+1)||I||^2}-3\rho_mI^m\big]||I||^2
\end{eqnarray}
On the other hand
\begin{eqnarray}\label{2}
\nonumber[h_{ij}I_k+h_{jk}I_i+h_{ki}J_j]'\!\!\!\!&=&\!\!\!\!\ -(\rho_ih_{jk}+\rho_jh_{ik}+\rho_kh_{ij})-2R[I_ig_{jk}+I_jg_{ik}+I_kg_{ij}]\\
\nonumber \!\!\!\!&-&\!\!\!\!\ 2[I_iRic_{jk}+I_jRic_{ik}+I_kRic_{ij}]\\
\nonumber \!\!\!\!&+&\!\!\!\!\ 2R[I_ih_{jk}+I_jh_{ik}+I_kh_{ij}]\\
\!\!\!\!&+&\!\!\!\!\ 2[ I_i\Lambda_{jk}+ I_j\Lambda_{ik}+I_k\Lambda_{ij}],
\end{eqnarray}
where $\Lambda_{jk}:=(Ric_{jr}\ell_k+Ric_{kr}\ell_j)\ell^r$. Multiplying (\ref{2}) with $I^i I^j I^k$ implies that
\begin{eqnarray}\label{3}
[h_{ij}I_k+h_{jk}I_i+h_{ki}J_j]'I^i I^j I^k= -3(\rho_mI^m+2Ric_{pq}I^pI^q)||I||^2.
\end{eqnarray}
On the other hand, since $p'+q'=0$ then we get
\begin{eqnarray}\label{4}
\Big[\frac{p'}{1+n}(h_{ij}I_k+h_{jk}I_i+h_{ki}J_j)+\frac{q'}{||I||^2}I_iI_jI_k\Big]I^i I^j I^k=\frac{nq'}{1+n}||I||^4
\end{eqnarray}
Putting  (\ref{1}), (\ref{3}) and (\ref{4}) in (\ref{mainCar}) implies that $C'_{ijk}I^iI^jI^k$ is a factor of $||I||^2$. More precisely, we have the following
\begin{eqnarray}
\nonumber C'_{ijk}I_iI_jI_k=\Big[\frac{nq'}{n+1}||I||^2    \!\!\!\!&-&\!\!\!\!\ q\Big(\frac{2(nq+1)(\rho^mI_m-Ric^{pq}I_pI_q)}{(n+1)||I||^2}-3\rho_mI^m\Big) \\
 \!\!\!\!&-&\!\!\!\!\ \frac{3p}{n+1}(\rho_mI^m+2Ric_{pq}I^pI^q) \Big]||I||^2.
\end{eqnarray}
This completes the proof.
\end{proof}

\bigskip

\noindent {\bf Proof of Theorem \ref{thm1}}: By Lemmas \ref{Lem1} and \ref{Lem2}, it follows that $R_{,i,j,k}I^i I^j I^k$ is a factor of $||I||^2$. Thus
\[
R_{,i,j,k}I^i I^j I^k=A_{ij}I_k+B_ig_{jk}.
\]
It is remarkable that, since  $R_{,i,j,k}$ is symmetric with respect to indexes i, j and k, then the order of indexes in this relation doesn't matter. Now, multiplying $R_{,i,j,k}$  with $y^k$ or $y^j$ implies that $R_{,i}=0$. It means that $R=R(x)$ and then $F_t$ is an Einstein metric.

\section{Normal Ricci Flow Equation on $(\alpha, \beta)$-Metrics}
If M is a compact manifold, then $S(M)$ is compact  and we can
normalize the Ricci flow equation by requiring that the flow keeps the volume of
$SM$ constant. Recalling the Hilbert form $\omega:= F_{y^i}dx^i$, that volume is
\[
Vol_{SM}:=\int_{SM}\frac{(-1)}{(n - 1)!}^{\frac{n(n-1)}{2}}\omega \wedge (d \omega)^{n-1}:=\int_{SM} dV_{SM}.
\]
During the evolution, $F$, $\omega$ and consequently the volume form $dV_{SM}$ and
the volume $Vol_{SM}$, all depend on $t$. On the other hand, the domain of
integration $SM$, being the quotient space of $TM_0$ under the equivalence
relation $z\sim y$ , $z= \lambda y$ for some $\lambda > 0$, is totally independent of any Finsler
metric, and hence does not depend on t. We have
\[
\frac{d}{dt}(dV_{SM})=\big[g_{ij} \frac{d}{dt}g_{ij} -n\frac{d}{dt}\log F\big] dV_{SM}
\]
A normalized Ricci flow for Finsler metrics is proposed by Bao as follows
\begin{equation}\label{logf}
\frac{d}{dt}\log F= -R +\frac{1}{Vol(SM)}\int_{SM}R\ dV,\ \ \  F(t=0)=F_0,
\end{equation}
where the underlying manifold $M$ is compact. Now, we let $Vol(SM)=1$. Then all of  Ricci-constant metrics are exactly the fixed points of the above flow.
Let
\[
Ric_{ij}=\frac{1}{2}(F^2R)_{.y^i.y^j}
\]
and differentiating (\ref{logf}) with respect to $y^i$ and $y^j$ the following normal Ricci flow tensor evaluation equation is concluded
\begin{equation}
\frac{d}{dt}g_{ij}=-2Ric_{ij}+\frac{2}{Vol(SM)}\int_{SM}R\ dVg_{ij},\ \ \  g(t=0)=g_0,
\end{equation}
Starting with any familiar metric on M as the initial data $F_0$, we may deform
it using the proposed normalized Ricci flow, in the hope of arriving at a Ricci constant
metric.

\begin{thm}
Let $(M, F)$ be a Finsler manifold of dimension $n\geq 3$. Suppose that $F=\Phi(\frac{\beta}{\alpha})\alpha$ be an $(\alpha, \beta)$-metric on $M$.  Then every deformation $F_t$ of the metric $F$  satisfying normal Ricci flow equation is an  Einstein metric.
\end{thm}
\begin{proof}
Now, we consider Finsler surfaces that satisfies the normal Ricci flow equation. Then
\begin{equation}
\frac{dg_{ij}}{dt}=-2Ric_{ij}+2\int_{SM} R\ dV g_{ij},  \ \ d \log F:=\frac{F'}{F}=-R+\int_{SM} R\ dV.
\end{equation}
By the same argument used in the un-normal Ricci flow case, we can calculate the variation of mean Cartan tensor as follows
\begin{eqnarray}
\nonumber I'_i=(g^{jk}C_{ijk})' \!\!\!\!&=&\!\!\!\!\ (g^{jk})'C_{ijk}+g^{jk}C'_{ijk}\\
\nonumber \!\!\!\!&=&\!\!\!\!\ 2[Ric^{jk}-\int_{SM} R\ dV g^{jk}]C_{ijk}+g^{jk}[Ric_{jk,i}+2\int_{SM} R\ dV C_{ijk}]\\
\nonumber \!\!\!\!&=&\!\!\!\!\ 2Ric^{jk}g_{jk,i}-(g^{jk}Ric_{jk})_{,i}+g^{jk}_{\ \ ,i}Ric_{jk}\\
\nonumber\!\!\!\!&=&\!\!\!\!\ -(g^{jk}Ric_{jk})_{,i}\\
\!\!\!\!&=&\!\!\!\!\ -\rho_{i}
\end{eqnarray}
Then we have
\begin{eqnarray}
\nonumber I'^i=(g^{ij}I_j)' \!\!\!\!&=&\!\!\!\!\ (g^{ij})'I_j+g^{ij}I'_j\\
\!\!\!\!&=&\!\!\!\!\ 2[Ric^{ij}-\int_{SM} R\ dV g^{ij}]I_j-g^{ij}\rho_{j}
\end{eqnarray}
By the same way that we used in un-normal Ricci flow, it follows that
\begin{eqnarray}
C'_{ijk}=\frac{-(I'^mI_m+I^mI'_m)}{||I||^2}C_{ijk} -\frac{1}{||I||^2}(\rho_{i}I_jI_k+\rho_{j}I_iI_k+\rho_{k}I_iI_j)
\end{eqnarray}
Contracting it with $I^iI^jI^k$ yields
\begin{equation}
C'_{ijk}I^i I^j I^k =( \Omega ||I||^2 -3\rho_{m}I^m)||I||^2,
\end{equation}
where
\[
\Omega:=-\frac{I'^mI_m+I^mI'_m}{||I||^2}=\frac{2\rho^mI_m-2Ric^{ml}I_l}{||I||^2}+2\int_{SM} R\ dV .
\]
By Lemma \ref{Lem1}, we deduce that $R_{,i,j,k}I^i I^j I^k$ is a factor of $||I||^2$. By the same argument, it result that every deformation $F_t$ of the metric $F$  satisfying normal Ricci flow equation is an  Einstein metric.
\end{proof}


Akbar Tayebi\\
Faculty  of Science, Department of Mathematics\\
Qom University\\
Qom. Iran\\
Email:\ akbar.tayebi@gmail

\bigskip

\noindent
Esmail Peyghan\\
Faculty  of Science, Department of Mathematics\\
Arak University\\
Arak. Iran\\
Email: epeyghan@gmail.com
\bigskip

\noindent
Behzad Najafi\\
Faculty  of Science, Department of Mathematics\\
Shahed University\\
Tehran. Iran\\
Email:\ najafi@shahed.ac.ir

\end{document}